\documentclass[twoside,reqno,10pt]{amsart}
\usepackage{latexsym}%%triangles:\lhd\rhd\unlhd\unrhd%
\newcommand{\qdn}{\hspace*{-1.5mm}}
\newcommand{\qqdn}{\hspace*{-2.5mm}}
\newcommand{\xqdn}{\hspace*{-5.0mm}}
\newcommand{\xxqdn}{\hspace*{-10mm}}

\newcommand{\fns}{\footnotesize}
\newcommand{\sst}{\scriptstyle}

\newcommand{\ffnk}[4]{\left[\qdn\ba{#1}#3\\#4\ea{\!\Big|\:#2}\right]}

\newcommand{\binm}{\binom}

\newcommand{\nnm}{\nonumber}
\newcommand{\be}{\begin{equation}}
\newcommand{\ee}{\end{equation}}
\newcommand{\ba}{\begin{array}}
\newcommand{\ea}{\end{array}}
\newcommand{\bmn}{\begin{eqnarray}}
\newcommand{\emn}{\end{eqnarray}}
\newcommand{\bnm}{\begin{eqnarray*}}
\newcommand{\enm}{\end{eqnarray*}}
\newcommand{\bln}{\begin{subequations}}
\newcommand{\eln}{\end{subequations}}

\newcommand{\lam}{\lambda}
\newtheorem{thm}{Theorem}
\newtheorem{lemm}[thm]{Lemma}
\newtheorem{corl}[thm]{Corollary}

\newtheorem{entry}{Entry}%%%%%%%%%%%%%%%%

\newcommand{\bbtm}[4]{\bibitem{kn:#1}{#2,}~{#3,}~{#4.}}
\newcommand{\cito}[1]{\cite{kn:#1}}
\newcommand{\citu}[2]{\cite[#2]{kn:#1}}

\begin{document} %%%%%%%%%% This paper is published in %%%%%%%
{\fns %\today\hfill%% \copyrightPrinted in China} %%%%%%%%%%%%%%%
%%%%%%%%%%%%%%%%%%%%%%%%%%%%%%%%%%%%%%%%%%%%%%%%%%%%%%%%%%%%%%
\title{Generalizations of Ramanujan's reciprocity formula and the Askey-Wilson integral}
\author{$^a$Chuanan Wei, $^b$Xiaoxia Wang, $^c$Qinglun Yan}
\dedicatory{
$^a$Department of Information Technology\\
  Hainan Medical College, Haikou 571199, China\\
  $^b$Department of Mathematics\\
  Shanghai University, Shanghai 200444, China\\
    $^c$College of Mathematics and Physics\\
    Nanjing University of Posts and Telecommunications,
    Nanjing 210046, China}
\thanks{\emph{Email addresses}:
      weichuanan78@163.com (C. Wei), xwang913@126.com (X. Wang),
      yanqinglun@yahoo.com.cn (Q. Yan)}

\address{ }


\begin{abstract}
By using two known transformation formulas for basic hypergeometric
series, we establish a direct extension of Bailey's
$_6\psi_6$-series identity. Subsequently, it and Milne's identity
are employed to drive multi-variable generalizations of Ramanujan's
reciprocity formula. Then we utilize also Milne's identity to deduce
a multi-variable generalization of the Askey-Wilson integral.
\end{abstract}

%%%%%%%%%%%%%%%%%%%%%%%%%%%%%%%%%%%%%%%%%%%%%%%%%%%%%%%%%%%%%%%%%%%
\maketitle\thispagestyle{empty}%%%%%%%%%%%%%%%%%%%%%%%%%%%%%%%%%%%%
\markboth{Chuanan Wei, Xiaoxia Wang, Qinglun Yan}%%%%%%%%%%%%%%%%%%%%%%%%%%%%
         {Generalizations of Ramanujan's reciprocity formula and the Askey-Wilson integral}

%%%%%%%%%%%%%%%%%%%%%%%%%%%%%%%%%%%%%%%%%%%%%%%%%%%%%%%%%%%%%%%%%%%
%%%%%%%%%%%%%%%%%%%%%%%%%%%%%%%%%%%%%%%%%%%%%%%%%%%%%%%%%%%%%%%%%%%
%%%%%%%%%%%%%%%%%%%%%%%%%%%%%%%%%%%%%%%%%%%%%%%%%%%%%%%%%%%%%%%%%%%
%%%%%%%%%%%%%%%%%%%%%%%%%%%%%%%%%%%%%%%%%%%%%%%%%%%%%%%%%%%%%%%%%%%
\section{Introduction}
%%%%%%%%%%%%%%%%%%%%%%%%%%%%%%%%%%%%%%%%%%%%%%%%%%%%%%%%%%%%%%%%%%%
%%%%%%%%%%%%%%%%%%%%%%%%%%%%%%%%%%%%%%%%%%%%%%%%%%%%%%%%%%%%%%%%%%%
For two complex numbers $q$ and $x$, define the $q$-shifted
factorial by \bnm
 (x;q)_n=
\begin{cases}
\prod_{i=0}^{n-1}(1-xq^i),&\quad\text{when}\quad n>0;\\
\quad1,&\quad\text{when}\quad n=0;\\
\frac{1}{\prod_{j=n}^{-1}(1-xq^j)},&\quad\text{when}\quad n<0.
\end{cases}
\enm
  The $q$-shifted factorial of infinite order reads as
\[(x;q)_{\infty}=\prod_{k=0}^{\infty}(1-xq^k) \quad\text{with}\quad |q|<1.\]
 For simplifying the expressions, we shall use the following compact
notations:
 \bnm
&&\,(a,b,\cdots,c;q)_n=(a;q)_n(b;q)_n\cdots(c;q)_n,\\
&&\,(a,b,\cdots,c;q) _{\infty}=(a;q)_{\infty}(b;q)_{\infty}\cdots(c;q)_{\infty},\\
&&\ffnk{cccc}{q}{a,b,\cdots,c}{\alpha,\beta,\cdots,\gamma}_n\!=
  \frac{(a;q)_n(b;q)_n\cdots(c;q)_n}{(\alpha;q)_n(\beta;q)_n\cdots(\gamma;q)_n},\\
&&\ffnk{cccc}{q}{a,b,\cdots,c}{\alpha,\beta,\cdots,\gamma}_{\infty}\!=
  \frac{(a;q)_{\infty}(b;q)_{\infty}\cdots(c;q)_{\infty}}
  {(\alpha;q)_{\infty}(\beta;q)_{\infty}\cdots(\gamma;q)_{\infty}}.
 \enm
Following Gasper and Rahman \cito{gasper}, define the unilateral
basic hypergeometric series and bilateral basic hypergeometric
series, respectively, by
 \bnm
&&\xqdn{_{1+r}\phi_s}\ffnk{cccccc}{q;z}{a_0,&a_1,&\cdots,a_r}{&b_1,&\cdots,b_s}
  =\sum_{k=0}^{\infty}\ffnk{cccccc}{q}{a_0,a_1,\cdots,a_r}{\,q,\:b_1,\:\cdots,b_s}_k
\Big\{(-1)^kq^{\binm{k}{2}}\Big\}^{s-r}z^k,\\
&&\qdn{_r\psi_s}\ffnk{cccccc}{q;z}{a_1,&\cdots,a_r}{b_1,&\cdots,b_s}
  =\sum_{k=-\infty}^{\infty}\!\ffnk{cccccc}{q}{a_1,\cdots,a_r}{b_1,\cdots,b_s}_k
\Big\{(-1)^kq^{\binm{k}{2}}\Big\}^{s-r}z^k.
 \enm
Then Watson's transformation formula (cf. \citu{gasper}{Equation
(2.5.1)}) and Bailey's $_6\psi_6$-series identity (cf.
\citu{gasper}{Equation (5.3.1)}}) can be stated, respectively, as
  \bmn
  &&{_8\phi_7}\ffnk{ccccccc}
 {q;\frac{q^{2+n}a^2}{bcde}}{a,q\sqrt{a},-q\sqrt{a},b,c,d,e,q^{-n}}
 {\sqrt{a},-\sqrt{a},qa/b,qa/c,qa/d,qa/e,q^{1+n}a}
  \nnm\\\label{watson}
  &&\:=\:\,\ffnk{ccccc}{q}{qa,qa/de}{qa/d,qa/e}_n
 {_4\phi_3}\ffnk{ccccccc}
 {q;q}{q^{-n},d,e,qa/bc}{qa/b,qa/c,deq^{-n}/a},
     \\\label{bailey}
  &&{_6\psi_6}\ffnk{cccccccccc}{q;\frac{qa^2}{bcde}}
 {a,q\sqrt{a},-q\sqrt{a},b,c,d,e}
 {\sqrt{a},-\sqrt{a},qa/b,qa/c,qa/d,qa/e}
 \nnm\\&&\:\,=\:
 \ffnk{ccccc}{q}{q,qa,q/a,qa/bc,qa/bd,qa/be,qa/cd,qa/ce,qa/de}
 {q/b,q/c,q/d,q/e,qa/b,qa/c,qa/d,qa/e,qa^2/bcde}_{\infty}
 \emn
provided $|qa^2/bcde|<1$.

In his lost notebook \citu{ramanujan}{p. 40}, Ramanujan recorded the
beautiful reciprocity formula:
 \bmn \label{Ramanujan}
\Big(1+\frac{1}{b}\Big)\sum_{k=0}^{\infty}\frac{(-a/b)^k}{(-qa;q)_k}\,q^{\binm{k+1}{2}}
-\text{idem}(a;b)
=\Big(\frac{1}{b}-\frac{1}{a}\Big)\frac{(q,qa/b,qb/a;q)_{\infty}}{(-qa,-qb;q)_{\infty}},
 \emn
where the symbol ``idem$(x,y)$'' after an expression means that the
preceding expression is repeated with $x$ and $y$ interchanged.

 The first published proof of \eqref{Ramanujan} is given
by Andrews \cito{andrews}. Other proofs can be found in
\cite{kn:adiga,kn:berndt,kn:bhargava-b,kn:kim,kn:somashekara}. In
the same paper, Andrews established the four-variable generalization
of \eqref{Ramanujan}:
 \bmn
&&\Big(1+\frac{1}{b}\Big)\sum_{k=0}^{\infty}\frac{(c,-qa/d;q)_k}{(-qa;q)_k(-c/b;q)_{k+1}}
\Big(-\frac{d}{b}\Big)^k-\text{idem}(a;b)
 \nnm\\\label{andrews}
&&\:=\:\Big(\frac{1}{b}-\frac{1}{a}\Big)
\ffnk{ccccc}{q}{q,qa/b,qb/a,c,d,cd/ab}
 {-qa,-qb,-c/a,-c/b,-d/a,-d/b}_{\infty}
 \emn
provided $\max\{|d/a|,|d/b|\}<1$. Different proofs of this identity
  can be seen in \cite{kn:kang,kn:liu-a}. The equivalent form
  of \eqref{andrews} due to Kang \citu{kang}{Theorem 1.2}
 reads as
\bnm
 &&\Big(1+\frac{1}{b}\Big)\sum_{k=0}^{\infty}
 \frac{(1+cdq^{2k}/b)(c,d,cd/ab;q)_k}{(-qa;q)_k(-c/b,-d/b;q)_{k+1}}\,q^{\binm{k+1}{2}}
\Big(-\frac{a}{b}\Big)^k -\text{idem}(a;b)\\
&&\:=\:\Big(\frac{1}{b}-\frac{1}{a}\Big)
\ffnk{ccccc}{q}{q,qa/b,qb/a,c,d,cd/ab}
 {-qa,-qb,-c/a,-c/b,-d/a,-d/b}_{\infty}.
 \enm
 The five-variable generalization of \eqref{Ramanujan} due to
Ma \citu{ma}{Theorem 1.3} (see also \citu{chu-a}{Theorem 5}) can be
expressed as
 \bmn
&&\sum_{k=0}^{\infty}\Big(1-\frac{aq^{2k+1}}{b}\Big)\frac{(-1/b;q)_{k+1}}{(-qa;q)_k}
  \frac{(-qa/c,-qa/d,-qa/e;q)_k}{(-c/b,-d/b,-e/b;q)_{k+1}}\bigg(\frac{cde}{qab}\bigg)^k
  -\text{idem}(a;b)
 \nnm\\\label{ma}
&&\:\:=\:\Big(\frac{1}{b}\!-\!\frac{1}{a}\Big)\!
\ffnk{ccccc}{q}{q,qa/b,qb/a,c,d,e,cd/ab,ce/ab,de/ab}
 {-qa,-qb,-c/a,-c/b,-d/a,-d/b,-e/a,-e/b,cde/qab}_{\infty}
 \emn
provided $|cde/qab|<1$. The equivalent form of \eqref{ma} due to Chu
and Zhang \citu{chu-a}{Theorem 10} reads as

\bnm
&&\quad\Big(1+\frac{1}{b}\Big)\sum_{k=0}^{\infty}\frac{(c,-qa/d,-qa/e;q)_kq^k}{(-qa,q^2ab/de;q)_k(-c/b;q)_{k+1}}
 -\text{idem}(a;b)\\
&&=\Big(\frac{1}{b}\!-\!\frac{1}{a}\Big)\!
\ffnk{ccccc}{q}{q,qa/b,qb/a,c,d,e,cd/ab,ce/ab,de/ab}
 {-qa,-qb,-c/a,-c/b,-d/a,-d/b,-e/a,-e/b,cde/qab}_{\infty}\\
&&+\:\frac{de}{qab}\Big(1+\frac{1}{b}\Big)\ffnk{cccccccccc}{q}{q,c,-qa/d,-qa/e,-de/b,-cde/ab^2}
  {-qa,-c/b,-d/b,-e/b,q^2ab/de,cde/qab}_{\infty}\\
  &&\times\:{_3\phi_2}\ffnk{cccccccccc}{q;q}{-d/b,-e/b,cde/qab}{-de/b,-cde/ab^2}-
 {_3\phi_2}\ffnk{cccccccccc}{q;q}{-d/a,-e/a,cde/qab}{-de/a,-cde/ba^2}\\
  &&\times\:\frac{de}{qab}\Big(1+\frac{1}{a}\Big)\ffnk{cccccccccc}{q}{q,c,-qb/d,-qb/e,-de/a,-cde/ba^2}
  {-qb,-c/a,-d/a,-e/a,q^2ab/de,cde/qab}_{\infty}.
  \enm

Define the $h$-function by
\[h(x;\lam)=(\lam e^{i\theta},\lam e^{-i\theta};q)_{\infty}
=\prod_{k=0}^{\infty}(1-2q^k\lam
x+q^{2k}\lam^2)\quad\text{with}\quad x=\cos\theta.\] For simplifying
the expressions, we shall use the following compact notation:
\[h(x;\alpha,\beta,\cdots,\gamma)=h(x;\alpha)h(x;\beta)\cdots h(x;\gamma).\]
Then the $q$-beta integral due to Askey and Wilson
\citu{askey-b}{Theorem 2.1} (see also \citu{andrews-r}{Chapter 10}
and \citu{gasper}{Chapter 6}) can be written as
 \bmn \label{askey-wilson}
\int_{0}^{\pi}\frac{h(\cos2\theta;1)}{h(\cos\theta;a,b,c,d)}\,d\theta
 =\frac{2\pi(abcd;q)_{\infty}}{(q,ab,ac,ad,bc,bd,cd;q)_{\infty}}
 \emn
provided $|abcd/q|<1$. All kinds of proofs of this integral can be
seen in
\cite{kn:askey-a,kn:bowman,kn:chu-b,kn:ismail-a,kn:ismail-b,kn:liu-b,kn:rahman-a}.
We refer to
\cite{kn:nassrallah,kn:rahman-b,kn:rahman-c,kn:rahman-d,kn:verma}
for its extensions.

The structure of the paper is arranged as follows. We shall use a
direct extension of Bailey's $_6\psi_6$-series identity and Milne's
identity to give multi-variable generalizations of Ramanujan's
reciprocity formula in Section 2. Then Milne's identity is employed
to offter a multi-variable generalization of the Askey-Wilson
integral in Section 3.

%%%%%%%%%%%%%%%%%%%%%%%%%%%%%%%%%%%%%%%%%%%%%%%%%%%%%%%%%%%%%%%%%%%
%%%%%%%%%%%%%%%%%%%%%%%%%%%%%%%%%%%%%%%%%%%%%%%%%%%%%%%%%%%%%%%%%%%
\section{Multi-variable generalizations of Ramanujan's reciprocity formula}
%%%%%%%%%%%%%%%%%%%%%%%%%%%%%%%%%%%%%%%%%%%%%%%%%%%%%%%%%%%%%%%%%%%
%%%%%%%%%%%%%%%%%%%%%%%%%%%%%%%%%%%%%%%%%%%%%%%%%%%%%%%%%%%%%%%%%%%

\begin{lemm}  \label{lemm-a}
For $\max\{|q^2a^3/bcdefg|,|qa/bc|\}<1$, there holds the
 direct generalization of Bailey's
$_6\psi_6$-series identity:
 \bnm
  &&\xqdn{_8\psi_8}\ffnk{ccccccc}
 {q;\frac{q^2a^3}{bcdefg}}{q\sqrt{a},-q\sqrt{a},b,c,d,e,f,g}{\sqrt{a},-\sqrt{a},qa/b,qa/c,qa/d,qa/e,qa/f,qa/g}
 \\&&\nnm\xqdn\:=\:\,
 \ffnk{ccccc}{q}{\sst q,qa,q/a,qa/bc,qa/bf,qa/cf,qa/df,qa/ef,qf/d,qf/e,g,g/a,q^2a^2/bdeg,q^2a^2/cdeg}
 {\sst q/b,q/c,q/d,q/e,q/f,qa/b,qa/c,qa/d,qa/e,qa/f,g/f,fg/a,q^2af/deg,q^2a^3/bcdefg}_{\infty}
 \\\nnm&&\xqdn\:\times\:
 {_8\phi_7}\ffnk{ccccccc}
 {q;\frac{qa}{bc}}{qaf/deg,q\sqrt{qaf/deg},-q\sqrt{qaf/deg},qa/de,qa/dg,qa/eg,bf/a,cf/a}
 {\sqrt{qaf/deg},-\sqrt{qaf/deg},qf/g,qf/e,qf/d,q^2a^2/bdeg,q^2a^2/cdeg}
  \\&&\nnm\xqdn\:+\:\,
 idem(f;g).
 \enm
\end{lemm}

\begin{proof}
Recall the transformation formula between two $_8\phi_7$-series (cf.
\citu{gasper}{Equation (2.10.1)}):
 \bnm
  &&\xqdn{_8\phi_7}\ffnk{ccccccc}
 {q;\frac{q^2a^2}{bcdef}}{a,q\sqrt{a},-q\sqrt{a},b,c,d,e,f}{\sqrt{a},-\sqrt{a},qa/b,qa/c,qa/d,qa/e,qa/f}
 \\&&\nnm\xqdn\:=\:\,
 \ffnk{ccccc}{q}{qa,qa/ef,q\lam/e,q\lam/f}{qa/e,qa/f,q\lam,q\lam/ef}_{\infty}
 \\\nnm&&\xqdn\:\times\:
 {_8\phi_7}\ffnk{ccccccc}
 {q;\frac{qa}{ef}}{\lam,q\sqrt{\lam},-q\sqrt{\lam},\lam b/a,\lam c/a,\lam d/a,e,f}
  {\sqrt{\lam},-\sqrt{\lam},qa/b,qa/c,qa/d,q\lam/e,q\lam/f},
 \enm
where $\lam=qa^2/bcd$ and $\max\{|qa/ef|, |\lam q/ef|\}<1$. Then we
have
 \bnm
  &&\xqdn {_8\phi_7}\ffnk{ccccccc}
 {q;\frac{q^2a^3}{bcdefg}}{f^2/a,qf/\sqrt{a},-qf/\sqrt{a},bf/a,cf/a,df/a,ef/a,gf/a}
 {f/\sqrt{a},-f/\sqrt{a},qf/b,qf/c,qf/d,qf/e,qf/g}
 \\&&\nnm\xqdn\:=\:\,
 \ffnk{ccccc}{q}{qf^2/a,qa/bc,q^2a^2/bdeg,q^2a^2/cdeg}{qf/b,qf/c,q^2af/deg,q^2a^3/bcdefg}_{\infty}
 \\\nnm&&\xqdn\:\times\:
{_8\phi_7}\ffnk{ccccccc}
 {q;\frac{qa}{bc}}{qaf/deg,q\sqrt{qaf/deg},-q\sqrt{qaf/deg},qa/de,qa/dg,qa/eg,bf/a,cf/a}
 {\sqrt{qaf/deg},-\sqrt{qaf/deg},qf/g,qf/e,qf/d,q^2a^2/bdeg,q^2a^2/cdeg},\\
 &&\xqdn {_8\phi_7}\ffnk{ccccccc}
 {q;\frac{q^2a^3}{bcdefg}}{g^2/a,qg/\sqrt{a},-qg/\sqrt{a},bg/a,cg/a,dg/a,eg/a,fg/a}
 {g/\sqrt{a},-g/\sqrt{a},qg/b,qg/c,qg/d,qg/e,qg/f}
 \\&&\nnm\xqdn\:=\:\,
 \ffnk{ccccc}{q}{qg^2/a,qa/bc,q^2a^2/bdef,q^2a^2/cdef}{qg/b,qg/c,q^2ag/def,q^2a^3/bcdefg}_{\infty}
 \\\nnm&&\xqdn\:\times\:
{_8\phi_7}\ffnk{ccccccc}
 {q;\frac{qa}{bc}}{qag/def,q\sqrt{qag/def},-q\sqrt{qag/def},qa/de,qa/df,qa/ef,bg/a,cg/a}
 {\sqrt{qag/def},-\sqrt{qag/def},qg/f,qg/e,qg/d,q^2a^2/bdef,q^2a^2/cdef}.
 \enm
Substituting the last two relations into the transformation formula
from a $_8\psi_8$-series to two $_8\phi_7$-series (cf.
\citu{gasper}{Equation (5.6.1)}):
  \bnm
  &&\xqdn{_8\psi_8}\ffnk{ccccccc}
 {q;\frac{q^2a^3}{bcdefg}}{q\sqrt{a},-q\sqrt{a},b,c,d,e,f,g}{\sqrt{a},-\sqrt{a},qa/b,qa/c,qa/d,qa/e,qa/f,qa/g}
 \\&&\nnm\xqdn\:=\:\,
 \ffnk{ccccc}{q}{q,qa,q/a,qa/bf,qa/cf,qa/df,qa/ef,qf/b,qf/c,qf/d,qf/e,g,g/a}
 {q/b,q/c,q/d,q/e,q/f,qa/b,qa/c,qa/d,qa/e,qa/f,g/f,fg/a,qf^2/a}_{\infty}
 \\\nnm&&\xqdn\:\times\:
 {_8\phi_7}\ffnk{ccccccc}
 {q;\frac{q^2a^3}{bcdefg}}{f^2/a,qf/\sqrt{a},-qf/\sqrt{a},bf/a,cf/a,df/a,ef/a,gf/a}
 {f/\sqrt{a},-f/\sqrt{a},qf/b,qf/c,qf/d,qf/e,qf/g}
  \\&&\nnm\xqdn\:+\:\,
 idem(f;g)
 \enm
provided $|q^2a^3/bcdefg|<1$, we obtain Lemma \ref{lemm-a} to
complete the proof.
\end{proof}

When $g=aq/e$, Lemma \ref{lemm-a} reduces to \eqref{bailey} under
the replacement $f\to e$ exactly.

\begin{thm}  \label{thm-a}
For $\max\{|cdefg/qa^2b^2|,|c|\}<1$, there holds the
 seven-variable generalization of Ramanujan's reciprocity formula:
 \bnm
\quad\rho(a,b;c,d,e,f,g)-\rho(b,a;c,d,e,f,g)=R(a,b,c,d,e;f,g)+R(a,b,c,d,e;g,f),
 \enm
where
 \bnm
  \rho(a,b;c,d,e,f,g)&&\xqdn=\:\frac{1}{b}\sum_{k=0}^{\infty}\Big(1-\frac{aq^{2k+1}}{b}\Big)\frac{(-1/b;q)_{k+1}}{(-qa;q)_k}
  \\&&\xqdn\times\:\frac{(-qa/c,-qa/d,-qa/e,-qa/f,-qa/g;q)_k}{(-c/b,-d/b,-e/b,-f/b,-g/b;q)_{k+1}}\Big(\frac{cdefg}{qa^2b^2}\Big)^k,
 \enm
 \bnm
 R(a,b,c,d,e;f,g)&&\xqdn=\:\frac{1}{g}\Big(\frac{1}{b}\!-\!\frac{1}{a}\Big)\!
 \ffnk{ccccc}{q}{q,c,f,qa/b,qb/a}{-qa,-qb,-c/a,-c/b,-d/a}_{\infty}
 \\&&\xxqdn\xxqdn\xxqdn\times\:
\ffnk{ccccc}{q}{-qa/g,-qb/g,qd/f,qe/f,cf/ab,df/ab,ef/ab,deg/ab,cdeg/a^2b^2}
{-d/b,-e/a,-e/b,-f/a,-f/b,f/g,qab/fg,qdeg/abf,cdefg/qa^2b^2}_{\infty}
 \\&&\xxqdn\xxqdn\xxqdn\times\:
{_8\phi_7}\ffnk{ccccccc}
 {q;c}{deg/abf,q\sqrt{deg/abf},-q\sqrt{deg/abf},de/ab,dg/ab,eg/ab,q/f,qab/cf}
 {\sqrt{deg/abf},-\sqrt{deg/abf},qg/f,qe/f,qd/f,deg/ab,cdeg/a^2b^2}.
 \enm
\end{thm}

\begin{proof}
Splitting the $_8\psi_8$-series in Lemma \ref{lemm-a} into two
parts, we get
 \bnm
&&{_8\psi_8}\ffnk{ccccccc}
 {q;\frac{q^2a^3}{bcdefg}}{q\sqrt{a},-q\sqrt{a},b,c,d,e,f,g}{\sqrt{a},-\sqrt{a},qa/b,qa/c,qa/d,qa/e,qa/f,qa/g}\\
&&\:\:=\:\sum_{k=0}^{\infty}\frac{1-aq^{2k}}{1-a}
\ffnk{ccccccc}{q}{b,c,d,e,f,g}{qa/b,qa/c,qa/d,qa/e,qa/f,qa/g}_k\bigg(\frac{q^2a^3}{bcdefg}\bigg)^k\\
&&\:\:+\:\qdn\sum_{k=-\infty}^{-1}\qdn\frac{1-aq^{2k}}{1-a}
\ffnk{ccccccc}{q}{b,c,d,e,f,g}{qa/b,qa/c,qa/d,qa/e,qa/f,qa/g}_k\bigg(\frac{q^2a^3}{bcdefg}\bigg)^k\\
&&\:\:=\:\sum_{k=0}^{\infty}\frac{1-aq^{2k}}{1-a}
\ffnk{ccccccc}{q}{b,c,d,e,f,g}{qa/b,qa/c,qa/d,qa/e,qa/f,qa/g}_k\bigg(\frac{q^2a^3}{bcdefg}\bigg)^k\\
&&\:\:-\:\frac{q^2(1-q^2/a)(1-a/b)(1-a/c)(1-a/d)(1-a/e)(1-a/f)(1-a/g)}{a^2(1-a)(1-q/b)(1-q/c)(1-q/d)(1-q/e)(1-q/f)(1-q/g)}\\
&&\:\:\times\:\sum_{k=0}^{\infty}\frac{1-q^{2k+2}/a}{1-q^2/a}
\ffnk{ccccccc}{q}{qb/a,qc/a,qd/a,qe/a,qf/a,qg/a}{q^2/b,q^2/c,q^2/d,q^2/e,q^2/f,q^2/g}_k\bigg(\frac{q^2a^3}{bcdefg}\bigg)^k.
 \enm
Combining the last relation with Lemma \ref{lemm-a} and then
replacing respectively the parameters $a,b,c,d,e,f,g$ by
$qa/b,-q/b,-qa/c,-qa/d,-qa/e,-qa/f,-qa/g$, we derive, after some
routine simplification, Theorem \ref{thm-a} to finish the proof.
\end{proof}

\begin{corl}  \label{corl-a}
For $|cde/abq^{n+1}|<1$, there holds the
 seven-variable generalization of Ramanujan's reciprocity formula:
 \bnm
&&\rho\,'(a,b;c,d,e,f,n)-\rho\,'(b,a;c,d,e,f,n)
\\&&=\Big(\frac{1}{b}\!-\!\frac{1}{a}\Big)\!
\ffnk{ccccc}{q}{q,qa/b,qb/a,c,d,e,cd/ab,ce/ab,de/ab}
 {-qa,-qb,-c/a,-c/b,-d/a,-d/b,-e/a,-e/b,cde/qab}_{\infty}
 \\&&\times\:\frac{fq^n}{ab}\frac{(qf/e,ef/ab;q)_n}{(-f/a,-f/b;q)_{n+1}}
 {_4\phi_3}\ffnk{ccccccc}
 {q;q}{q^{-n},q/e,qab/ce,qab/de}{q^{1-n}ab/ef,qf/e,q^2ab/cde},
 \enm
where
 \bnm
  \rho\,'(a,b;c,d,e,f,n)&&\xqdn=\:\frac{1}{b}\sum_{k=0}^{\infty}\Big(1-\frac{aq^{2k+1}}{b}\Big)\frac{(-1/b;q)_{k+1}}{(-qa;q)_k}
  \\&&\xqdn\times\:\frac{(-qa/c,-qa/d,-qa/e,-qa/f,-q^{1+n}f/b;q)_k}{(-c/b,-d/b,-e/b,-f/b,-a/fq^n;q)_{k+1}}\Big(\frac{cde}{abq^{n+1}}\Big)^k.
  \enm
\end{corl}

\begin{proof}
Performing the replacements $e\to f, f\to e, g\to ab/fq^n$ in
Theorem \ref{thm-a}, we attain
 \bmn
  &&\xxqdn\rho(a,b;c,d,e,f,ab/fq^n)-\rho(b,a;c,d,e,f,ab/fq^n)
 \nnm\\\nnm&&\xxqdn=\Big(\frac{1}{b}\!-\!\frac{1}{a}\Big)\!
\ffnk{ccccc}{q}{q,qa/b,qb/a,c,d,e,cd/ab,ce/ab,de/ab}
 {-qa,-qb,-c/a,-c/b,-d/a,-d/b,-e/a,-e/b,cde/qab}_{\infty}
 \emn
 \bmn
&&\xxqdn\times\:\frac{fq^n}{ab}\frac{(qf/e,ef/ab;q)_n}{(-f/a,-f/b;q)_{n+1}}
\ffnk{ccccc}{q}{q/d,qab/cd}{e/d,q^2ab/cde}_n
  \nnm\\&&\xxqdn\times\:
{_8\phi_7}\ffnk{ccccccc}
 {q;c}{d/eq^n,q\sqrt{d/eq^n},-q\sqrt{d/eq^n},d/fq^n,q/e,df/ab,qab/ce,q^{-n}}
 {\sqrt{d/eq^n},-\sqrt{d/eq^n},qf/e,d/q^n,q^{1-n}ab/ef,cd/abq^n,qd/e}.
 \label{equation-a}
 \emn
According to \eqref{watson}, there holds the relation:
 \bnm
  &&\xqdn{_8\phi_7}\ffnk{ccccccc}
 {q;c}{d/eq^n,q\sqrt{d/eq^n},-q\sqrt{d/eq^n},d/fq^n,q/e,df/ab,qab/ce,q^{-n}}
 {\sqrt{d/eq^n},-\sqrt{d/eq^n},qf/e,d/q^n,q^{1-n}ab/ef,cd/abq^n,qd/e}
 \nnm\\&&\xqdn\:=\:\,
 \ffnk{ccccc}{q}{e/d,q^2ab/cde}{q/d,qab/cd}_n
{_4\phi_3}\ffnk{ccccccc}
 {q;q}{q^{-n},q/e,qab/ce,qab/de}{q^{1-n}ab/ef,qf/e,q^2ab/cde}.
 \enm
Combing \eqref{equation-a} with the last equation, we achieve
Corollary \ref{corl-a} to complete the proof.
\end{proof}

When $g=ab/e$, Theorem \ref{thm-a} reduces to \eqref{ma} under the
replacement $f\to e$. When $n=0$, Corollary \ref{corl-a} also
reduces to \eqref{ma}.

Employing the substitutions $a\to -b/xy,b\to-b/y,c\to bc/xy^2,d\to
bd/xy^2,e\to b,f\to be/xy^2,g\to bf/xy^2$ in Theorem \ref{thm-a} and
appealing to the relation:
 \bmn\label{relation}
(q/a;q)_k=(-1)^kq^{\binm{k+1}{2}}a^{-k}(q^{-k}a;q)_k,
 \emn
we obtain the following transformation formula.

\begin{thm}  \label{thm-b}
For $\max\{|bcdef/qx^2y^4|,|bc/xy^2|\}<1$, there holds the
 seven-variable generalization of Jacobi's triple product identity and the quintuple product identity:
 \bnm
&&\qdn\sum_{k=0}^{\infty}(1-q^{2k+1}/x)
\frac{\sst(q/xy,q^{-k}b/y,q^{-k}c/y,q^{-k}d/y,q^{-k}e/y,q^{-k}f/y;q)_k}
{\sst(y,b/xy,c/xy,d/xy,e/xy,f/xy;q)_{k+1}}\,q^{\frac{k(5k+3)}{2}}(-y/x^2)^k\\
&&\xqdn-\:x^2\sum_{k=0}^{\infty}(1-q^{2k+1}x)
\frac{\sst(q/y,q^{-k}b/xy,q^{-k}c/xy,q^{-k}d/xy,q^{-k}e/xy,q^{-k}f/xy;q)_k}
{\sst(xy,b/y,c/y,d/y,e/y,f/y;q)_{k+1}}\,q^{\frac{k(5k+3)}{2}}(-x^3y)^k\\
 &&\xqdn=S(x,y,b,c,d;e,f)+S(x,y,b,c,d;f,e),
 \enm
where
 \bnm
 S(x,y,b,c,d;e,f)&&\xqdn=\:\frac{x^2y}{f}
 \ffnk{ccccc}{q}{q,qx,1/x,e,qd/e}{y,b/y,c/y,d/y,e/y}_{\infty}
 \\&&\xxqdn\xxqdn\xxqdn\times\:
\ffnk{ccccc}{q}{qy/f,qxy/f,qxy^2/e,bc/xy^2,be/xy^2,ce/xy^2,de/xy^2,bdf/xy^2,cdf/xy^2}
{xy,b/xy,c/xy,d/xy,e/xy,e/f,qdf/e,qxy^2/ef,bcdef/qx^2y^4}_{\infty}
 \\&&\xxqdn\xxqdn\xxqdn\times\:
{_8\phi_7}\ffnk{ccccccc}
 {q;\frac{bc}{xy^2}}{df/e,q\sqrt{df/e},-q\sqrt{df/e},d,f,df/xy^2,qxy^2/be,qxy^2/ce}
 {\sqrt{df/e},-\sqrt{df/e},qf/e,qd/e,qxy^2/e,bdf/xy^2,cdf/xy^2}.
 \enm
\end{thm}

Theorem \ref{thm-b} contain several known results as special cases.
The details are displayed as follows:
 \bnm
  &&\text{When}\:\: b,c,e\to0,f\to xy^2/d,\:\text{Theorem \ref{thm-b} reduces to \citu{berndt}{Theorem 3.1}};\\
  &&\text{When}\:\: b,e\to0,f\to xy^2/d, \:\text{Theorem \ref{thm-b} reduces to the main result of \cito{bhargava-a}};\\
 &&\text{When}\:\: c\to0, e\to d,f\to xy^2/d,\: \text{Theorem \ref{thm-b} reduces to \citu{kang}{Theorem 6.1}};\\
 &&\text{When}\:\: e\to d,f\to xy^2/d, \:\text{Theorem \ref{thm-b} reduces to \citu{chu-a}{Theorem 9}}.
 \enm

Performing the replacements $a\to -b/xy,b\to-b/y,c\to bc/xy^2,d\to
bd/xy^2,e\to b,f\to be/xy^2$ in Corollary \ref{corl-a} and utilizing
\eqref{relation}, we get the following transformation formula.

\begin{corl}  \label{corl-b}
For $|bcd/xy^2q^{n+1}|<1$, there holds the
 seven-variable generalization of Jacobi's triple product identity and the quintuple product identity:
 \bnm
&&\qdn\sum_{k=0}^{\infty}(1-q^{2k+1}/x)
\frac{\sst(q/xy,q^{-k}b/y,q^{-k}c/y,q^{-k}d/y,q^{-k}e/y,q^{-k-n}xy/e;q)_k}
{\sst(y,b/xy,c/xy,d/xy,e/xy,q^{-n}y/e;q)_{k+1}}\,q^{\frac{k(5k+3)}{2}}(-y/x^2)^k\\
&&\xqdn-\:x^2\sum_{k=0}^{\infty}(1-q^{2k+1}x)
\frac{\sst(q/y,q^{-k}b/xy,q^{-k}c/xy,q^{-k}d/xy,q^{-k}e/xy,q^{-k-n}y/e;q)_k}
{\sst(xy,b/y,c/y,d/y,e/y,q^{-n}xy/e;q)_{k+1}}\,q^{\frac{k(5k+3)}{2}}(-x^3y)^k\\
 \\&&\xqdn=
\ffnk{ccccc}{q}{q,x,q/x,b,c,d,bc/xy^2,bd/xy^2,cd/xy^2}
 {y,xy,b/y,c/y,d/y,b/xy,c/xy,d/xy,bcd/qxy^2}_{\infty}
 \\&&\xqdn\times\:\frac{eq^n}{(-y)}\frac{(e,qe/xy^2;q)_n}{(e/y,e/xy;q)_{n+1}}
 {_4\phi_3}\ffnk{ccccccc}
 {q;q}{q^{-n},q/b,q/c,q/d}{q^{1-n}/e,qe/xy^2,q^2xy^2/bcd}.
 \enm
\end{corl}

When $n=0$, Corollary \ref{corl-b} also reduces to
\citu{chu-a}{Theorem 9}.

A nontrivial extension of \eqref{bailey} due to Milne
\citu{milne}{Theorem 1.7} can be stated as follows.

\begin{lemm}  \label{lemm-b}
For $x_\iota y_\iota=q^{1+N_\iota}a$ and $|q^{1-N}a^2/bcde|<1$ with
$N_\iota\in\mathbb{N}_0$ and $N=\sum_{\iota=1}^nN_\iota$, there
holds the bilateral series identity:
 \bnm
  &&\xqdn{_{2n+6}\psi_{2n+6}}\ffnk{ccccccc}
 {q;\frac{q^{1-N}a^2}{bcde}}{q\sqrt{a},-q\sqrt{a},b,c,d,e,\{x_\iota,y_\iota\}_{\iota=1}^n}
  {\sqrt{a},-\sqrt{a},qa/b,qa/c,qa/d,qa/e,\{qa/x_\iota,qa/y_\iota\}_{\iota=1}^n}
 \\&&\xqdn\:=\:\,
  \ffnk{ccccc}{q}{\sst q,qa,q/a,qa/bc,qa/bd,qa/be,qa/cd,qa/ce,qa/de}
 {\sst q/b,q/c,q/d,q/e,qa/b,qa/c,qa/d,qa/e,qa^2/bcde}_{\infty}
 \prod_{\iota=1}^n\ffnk{ccccc}{q}{\sst x_\iota,x_\iota/a,qe/y_\iota,qa/ey_\iota}
 {\sst x_\iota/e,ex_\iota/a,q/y_\iota,qa/y_\iota}_{\infty}
  \\&&\xqdn\:\times\:\, \sum_{\tilde{m}}\ffnk{ccccc}{q}{qa/x_ny_n}{q}_{m_n}
  \ffnk{ccccc}{q}{be/a,ce/a,de/a}{qe/x_n,qe/y_n,bcde/a^2}_{M_n}q^{M_n}
 \\&&\xqdn\:\times\:
 \prod_{s=1}^{n-1}\ffnk{ccccc}{q}{qa/x_sy_s}{q}_{m_s}
  \ffnk{ccccc}{q}{ex_{s+1}/a,ey_{s+1}/a}{qe/x_s,qe/y_s}_{M_s}
  \bigg(\frac{qa}{x_{s+1}y_{s+1}}\bigg)^{M_s},
 \enm
where the multiple sum runs over $\tilde{m}=(m_1,m_2,\cdots,m_n)\in
\mathbb{N}_{0}^n$ with their partial sums denoted by
$M_s=\sum_{\iota=1}^sm_s$.
\end{lemm}

\begin{thm}  \label{thm-c}
For $ab/x_\iota y_\iota=q^{N_\iota}$ and $|cde/abq^{N+1}|<1$ with
$N_\iota\in\mathbb{N}_0$ and $N=\sum_{\iota=1}^nN_\iota$, there
holds the multi-variable generalization of Ramanujan's reciprocity
formula:
 \bnm
&&\:\,\rho(a,b;c,d,e,\{x_\iota,y_\iota\}_{\iota=1}^n)-\rho(b,a;c,d,e,\{x_\iota,y_\iota\}_{\iota=1}^n)
 \\&&\nnm\xqdn\:=\:\,
  \Big(\frac{1}{b}\!-\!\frac{1}{a}\Big)\!
\ffnk{ccccc}{q}{q,qa/b,qb/a,c,d,e,cd/ab,ce/ab,de/ab}
 {-qa,-qb,-c/a,-c/b,-d/a,-d/b,-e/a,-e/b,cde/qab}_{\infty}
 \\&&\nnm\xqdn\:\times\:\,
  \prod_{\iota=1}^n\frac{1}{x_\iota}\ffnk{ccccc}{q}{-qa/x_\iota,-qb/x_\iota,qy_\iota/e,ey_\iota/ab}
 {e/x_\iota,qab/ex_\iota,-y_\iota/a,-y_\iota/b}_{\infty}
  \\&&\nnm\xqdn\:\times\:\,
  \sum_{\tilde{m}}\ffnk{ccccc}{q}{x_ny_n/ab}{q}_{m_n}
  \ffnk{ccccc}{q}{q/e,qab/ce,qab/de}{qx_n/e,qy_n/e,q^2ab/cde}_{M_n}q^{M_n}
 \\\nnm&&\xqdn\:\times\:
 \prod_{s=1}^{n-1}\ffnk{ccccc}{q}{x_sy_s/ab}{q}_{m_s}
  \ffnk{ccccc}{q}{qab/ex_{s+1},qab/ey_{s+1}}{qx_s/e,qy_s/e}_{M_s}
  \bigg(\frac{x_{s+1}y_{s+1}}{ab}\bigg)^{M_s},
 \enm
 where
 \bnm
  \rho(a,b;c,d,e,\{x_\iota,y_\iota\}_{\iota=1}^n)&&\xqdn\!\!=\frac{1}{b^n}
  \sum_{k=0}^{\infty}\Big(1-\frac{aq^{2k+1}}{b}\Big)\frac{(-1/b;q)_{k+1}}{(-qa;q)_k}
  \\&&\!\!\xqdn\times\:\frac{(-qa/c,-qa/d,-qa/e,\{-qa/x_\iota,-qa/y_\iota\}_{\iota=1}^n;q)_k}
  {(-c/b,-d/b,-e/b,\{-x_\iota/b,-y_\iota/b\}_{\iota=1}^n;q)_{k+1}}\Big(\frac{cde}{abq^{N+1}}\Big)^k
  \enm
 and the multiple sum runs over
$\tilde{m}=(m_1,m_2,\cdots,m_n)\in \mathbb{N}_{0}^n$ with their
partial sums denoted by $M_s=\sum_{\iota=1}^sm_s$.
\end{thm}

\begin{proof}
Splitting the $_{2n+6}\psi_{2n+6}$-series in Lemma \ref{lemm-b} into
two parts, we attain
 \bnm
&&\xqdn\qqdn{_{2n+6}\psi_{2n+6}}\ffnk{ccccccc}
 {q;\frac{q^{1-N}a^2}{bcde}}{q\sqrt{a},-q\sqrt{a},b,c,d,e,\{x_\iota,y_\iota\}_{\iota=1}^n}
  {\sqrt{a},-\sqrt{a},qa/b,qa/c,qa/d,qa/e,\{qa/x_\iota,qa/y_\iota\}_{\iota=1}^n}\\
&&\xqdn\qqdn\:\:=\:\sum_{k=0}^{\infty}\frac{1-aq^{2k}}{1-a}
\ffnk{ccccccc}{q}{b,c,d,e,\{x_\iota,y_\iota\}_{\iota=1}^n}
{qa/b,qa/c,qa/d,qa/e,\{qa/x_\iota,qa/y_\iota\}_{\iota=1}^n}_k\bigg(\frac{q^{1-N}a^2}{bcde}\bigg)^k\\
&&\xqdn\qqdn\:\:+\:\qdn\sum_{k=-\infty}^{-1}\qdn\frac{1-aq^{2k}}{1-a}
\ffnk{ccccccc}{q}{b,c,d,e,\{x_\iota,y_\iota\}_{\iota=1}^n}
 {qa/b,qa/c,qa/d,qa/e,\{qa/x_\iota,qa/y_\iota\}_{\iota=1}^n}_k\bigg(\frac{q^{1-N}a^2}{bcde}\bigg)^k\\
&&\xqdn\qqdn\:\:=\:\sum_{k=0}^{\infty}\frac{1-aq^{2k}}{1-a}
\ffnk{ccccccc}{q}{b,c,d,e,\{x_\iota,y_\iota\}_{\iota=1}^n}
 {qa/b,qa/c,qa/d,qa/e,\{qa/x_\iota,qa/y_\iota\}_{\iota=1}^n}_k\bigg(\frac{q^{1-N}a^2}{bcde}\bigg)^k\\
&&\xqdn\qqdn\:\:-\:\frac{(1-q^2/a)(1-a/b)(1-a/c)(1-a/d)(1-a/e)}{(1-a)(1-q/b)(1-q/c)(1-q/d)(1-q/e)}
 \bigg(\frac{q}{a}\bigg)^{n+1}\prod_{\iota=1}^n\frac{(1-a/x_\iota)(1-a/y_\iota)}{(1-q/x_\iota)(1-q/y_\iota)}\\
&&\xqdn\qqdn\:\:\times\:\sum_{k=0}^{\infty}\frac{1-q^{2k+2}/a}{1-q^2/a}
\ffnk{ccccccc}{q}{qb/a,qc/a,qd/a,qe/a,\{qx_\iota/a,qy_\iota/a\}_{\iota=1}^n}
 {q^2/b,q^2/c,q^2/d,q^2/e,\{q^2/x_\iota,q^2/y_\iota\}_{\iota=1}^n}_k\bigg(\frac{q^{1-N}a^2}{bcde}\bigg)^k.
 \enm
Combining the last relation with Lemma \ref{lemm-b} and then
replacing respectively the parameters
$a,b,c,d,e,\{x_\iota,y_\iota\}_{\iota=1}^n$ by
$qa/b,-q/b,-qa/c,-qa/d,-qa/e,\{-qa/x_\iota,-qa/y_\iota\}_{\iota=1}^n$,
we achieve Theorem \ref{thm-c} to complete the proof.
\end{proof}

When $n=0$, Theorem \ref{thm-c} reduces to \eqref{ma}. When $n=1$,
Theorem \ref{thm-c} reduces to Corollary \ref{corl-a} under the
replacement $x_1\to f, N_1\to n$.

Employing the substitutions $a\to-b/xy,b\to-b/y,c\to bc/xy^2,d\to
bd/xy^2,e\to b,\{x_\iota,y_\iota\}_{\iota=1}^n \to
\{bx_\iota/xy^2,by_\iota/xy^2\}_{\iota=1}^n$ in Theorem \ref{thm-c}
and using \eqref{relation}, we obtain the following transformation
formula.

\begin{thm}  \label{thm-d}
For $xy^2/x_\iota y_\iota=q^{N_\iota}$ and $|bcd/xy^2q^{N+1}|<1$
with $N_\iota\in\mathbb{N}_0$ and $N=\sum_{\iota=1}^nN_\iota$, there
holds the multi-variable generalization of Jacobi's triple product
identity and the quintuple product identity:
 \bnm
 &&\qdn\sum_{k=0}^{\infty}(1-q^{2k+1}/x)
\frac{(q/xy,q^{-k}b/y,q^{-k}c/y,q^{-k}d/y,\{q^{-k}x_\iota/y,q^{-k}y_\iota/y\}_{\iota=1}^n;q)_k}
{(y,b/xy,c/xy,d/xy,\{x_\iota/xy,y_\iota/xy\}_{\iota=1}^n;q)_{k+1}}
 \\&&\times\:q^{\frac{(2n+3)k^2+(2n+1)k}{2}}\,(-y/x^{n+1})^k\\
 &&-\:x^{n+1}\sum_{k=0}^{\infty}(1-q^{2k+1}x)
\frac{\sst(q/y,q^{-k}b/xy,q^{-k}c/xy,q^{-k}d/xy,\{q^{-k}x_\iota/xy,q^{-k}y_\iota/xy\}_{\iota=1}^n;q)_k}
{(xy,b/y,c/y,d/y,\{x_\iota/y,y_\iota/y\}_{\iota=1}^n;q)_{k+1}}
 \\&&\times\:q^{\frac{(2n+3)k^2+(2n+1)k}{2}}\,(-x^{n+2}y)^k
 \\&&=\:\,
 (-xy)^n\ffnk{ccccc}{q}{q,x,q/x,b,c,d,bc/xy^2,bd/xy^2,cd/xy^2}
 {y,xy,b/y,c/y,d/y,b/xy,c/xy,d/xy,bcd/qxy^2}_{\infty}\\
 &&\times\:\,
  \prod_{\iota=1}^n\frac{1}{x_\iota}\ffnk{ccccc}{q}{qy/x_\iota,qxy/x_\iota,y_\iota,qy_\iota/xy^2}
 {q/x_\iota,xy^2/x_\iota,y_\iota/y,y_\iota/xy}_{\infty}
  \\&&\times\:\,
  \sum_{\tilde{m}}\ffnk{ccccc}{q}{x_ny_n/xy^2}{q}_{m_n}
  \ffnk{ccccc}{q}{q/b,q/c,q/d}{qx_n/xy^2,qy_n/xy^2,q^2xy^2/bcd}_{M_n}q^{M_n}
 \enm
 \bnm
&&\times\:
 \prod_{s=1}^{n-1}\ffnk{ccccc}{q}{x_sy_s/xy^2}{q}_{m_s}
  \ffnk{ccccc}{q}{q/x_{s+1},q/y_{s+1}}{qx_s/xy^2,qy_s/xy^2}_{M_s}
  \bigg(\frac{x_{s+1}y_{s+1}}{xy^2}\bigg)^{M_s},
 \enm
where the multiple sum runs over $\tilde{m}=(m_1,m_2,\cdots,m_n)\in
\mathbb{N}_{0}^n$ with their partial sums denoted by
$M_s=\sum_{\iota=1}^sm_s$.
\end{thm}

When $n=0$, Theorem \ref{thm-d} reduces to \citu{chu-a}{Theorem 9}.
When $n=1$, Theorem \ref{thm-d} reduces to Corollary \ref{corl-b}
under the replacement $x_1\to f, N_1\to n$.
%%%%%%%%%%%%%%%%%%%%%%%%%%%%%%%%%%%%%%%%%%%%%%%%%%%%%%%%%%%%%%%%%%%
%%%%%%%%%%%%%%%%%%%%%%%%%%%%%%%%%%%%%%%%%%%%%%%%%%%%%%%%%%%%%%%%%%%
\section{A multi-variable generalization of the
Askey-Wilson integral}
%%%%%%%%%%%%%%%%%%%%%%%%%%%%%%%%%%%%%%%%%%%%%%%%%%%%%%%%%%%%%%%%%%%
%%%%%%%%%%%%%%%%%%%%%%%%%%%%%%%%%%%%%%%%%%%%%%%%%%%%%%%%%%%%%%%%%%%
\begin{thm}  \label{thm-e}
For $u_\iota/v_\iota=q^{N_\iota}$ and $|abcd/q^{N+1}|<1$ with
$N_\iota\in\mathbb{N}_0$ and $N=\sum_{\iota=1}^nN_\iota$, there
holds the multi-variable generalization of the Askey-Wilson
integral:
 \bnm
\int_{0}^{\pi}\frac{h(\cos2\theta;1)}{h(\cos\theta;a,b,c,d)}
\prod_{\iota=1}^n\frac{h(\cos\theta;u_\iota)}{h(\cos\theta;v_\iota)}\,d\theta
&&\xqdn=\:\frac{2\pi}{1-abcd/q^{N+1}}\frac{(abcd/q;q)_{\infty}}{(q,ab,ac,ad,bc,bd,cd;q)_{\infty}}
\\&&\xqdn\times\:\prod_{\iota=1}^n\ffnk{ccccc}{q}{du_\iota,qu_\iota/d}{dv_\iota,qv_\iota/d}_{\infty}
\Bigg/\Omega(a,b,c,d,n),
 \enm
 where
 \bnm
  \Omega(a,b,c,d,n)&&\xqdn\!\!=\sum_{\tilde{m}}\ffnk{ccccc}{q}{v_n/u_n}{q}_{m_n}
  \ffnk{ccccc}{q}{q/ad,q/bd,q/cd}{q/du_n,qv_n/d,q^2/abcd}_{M_n}q^{M_n}
 \\&&\xqdn\:\times\:
 \prod_{s=1}^{n-1}\ffnk{ccccc}{q}{v_s/u_s}{q}_{m_s}
  \ffnk{ccccc}{q}{qu_{s+1}/d,q/dv_{s+1}}{q/du_s,qv_s/d}_{M_s}
  \bigg(\frac{v_{s+1}}{u_{s+1}}\bigg)^{M_s}
  \enm
 and the multiple sum runs over
$\tilde{m}=(m_1,m_2,\cdots,m_n)\in \mathbb{N}_{0}^n$ with their
partial sums denoted by $M_s=\sum_{\iota=1}^sm_s$.
\end{thm}

\begin{proof}
Performing the replacements $a\to qz^2,b\to qz/a,c\to qz/b, d\to
qz/c, e\to
qz/d,\{x_\iota,y_\iota\}_{\iota=1}^n\to\{qzu_\iota,qz/v_\iota\}_{\iota=1}^n$
in Lemma \ref{lemm-b} and then multiplying across that equation by
\[\frac{(1-z^2)(1-qz^2)}{(1-az)(1-bz)(1-cz)(1-dz)}\prod_{\iota=1}^n\frac{1-zu_\iota}{1-zv_\iota},\]
we get
 \bmn
&&\xqdn\frac{(1-z^2)(1-qz^2)}{(1-az)(1-bz)(1-cz)(1-dz)}\prod_{\iota=1}^n\frac{1-zu_\iota}{1-zv_\iota}
 \nnm\\\nnm&&\xqdn\:\times\:{_{2n+6}\psi_{2n+6}}\ffnk{ccccccc}
 {q;\frac{abcd}{q^{N+1}}}{q^{\frac{3}{2}}z,-q^{\frac{3}{2}}z,qz/a,qz/b,qz/c,qz/d,\{qzu_\iota,qz/v_\iota\}_{\iota=1}^n}
 {q^{\frac{1}{2}}z,-q^{\frac{1}{2}}z,qaz,qbz,qcz,qdz,\{qz/u_\iota,qzv_\iota\}_{\iota=1}^n}\\
&&\xqdn=\frac{(z^2,1/z^2;q)_{\infty}}{(az,a/z,bz,b/z,cz,c/z,dz,d/z;q)_{\infty}}
\prod_{\iota=1}^n\ffnk{ccccc}{q}{zu_\iota,u_\iota/z}{zv_\iota,v_\iota/z}_{\infty}
 \nnm \\\nnm&&\xqdn\times\:
\frac{(q,ab,ac,ad,bc,bd,cd;q)_{\infty}}{(abcd/q;q)_{\infty}}
\prod_{\iota=1}^n\ffnk{ccccc}{q}{dv_\iota,qv_\iota/d}{du_\iota,qu_\iota/d}_{\infty}
 \emn
 \bmn
 &&\xqdn\times\: \sum_{\tilde{m}}\ffnk{ccccc}{q}{v_n/u_n}{q}_{m_n}
  \ffnk{ccccc}{q}{q/ad,q/bd,q/cd}{q/du_n,qv_n/d,q^2/abcd}_{M_n}q^{M_n}
 \nnm\\&&\xqdn\times\:\label{integral-before}
 \prod_{s=1}^{n-1}\ffnk{ccccc}{q}{v_s/u_s}{q}_{m_s}
  \ffnk{ccccc}{q}{qu_{s+1}/d,q/dv_{s+1}}{q/du_s,qv_s/d}_{M_s}
  \bigg(\frac{v_{s+1}}{u_{s+1}}\bigg)^{M_s}
 \emn
provided  $u_\iota/v_\iota=q^{N_\iota}$ and $|abcd/q^{N+1}|<1$. The
expression on the left hand side of \eqref{integral-before} can be
reformulated as
 \bmn
&&\xqdn\xqdn\qdn
f(z)=\frac{(1-z^2)(1-qz^2)}{(1-az)(1-bz)(1-cz)(1-dz)}\prod_{\iota=1}^n\frac{1-zu_\iota}{1-zv_\iota}
 \nnm\\\nnm&&\xqdn\:\times\:\sum_{k=0}^{\infty}\frac{1-q^{2k+1}z^2}{1-qz^2}
\ffnk{ccccccc}{q}{qz/a,qz/b,qz/c,qz/d,\{qzu_\iota,qz/v_\iota\}_{\iota=1}^n}
{qaz,qbz,qcz,qdz,\{qz/u_\iota,qzv_\iota\}_{\iota=1}^n}_k\bigg(\frac{abcd}{q^{N+1}}\bigg)^k\\
&&\xqdn\:+\:\frac{(1-z^2)(1-qz^2)}{(1-az)(1-bz)(1-cz)(1-dz)}\prod_{\iota=1}^n\frac{1-zu_\iota}{1-zv_\iota}
 \nnm\\\nnm&&\xqdn\:\times\sum_{k=-\infty}^{-1}\qdn\frac{1-q^{2k+1}z^2}{1-qz^2}
\ffnk{ccccccc}{q}{qz/a,qz/b,qz/c,qz/d,\{qzu_\iota,qz/v_\iota\}_{\iota=1}^n}
{qaz,qbz,qcz,qdz,\{qz/u_\iota,qzv_\iota\}_{\iota=1}^n}_k\bigg(\frac{abcd}{q^{N+1}}\bigg)^k\\
&&\xqdn\:=\sum_{k=0}^{\infty}(1-z^2)(1-q^{2k+1}z^2)
 \frac{(qz/a,qz/b,qz/c,qz/d;q)_k}{(az,bz,cz,dz;q)_{k+1}}\nnm\\\nnm
 &&\xqdn\:\times\prod_{\iota=1}^n\frac{(zu_\iota;q)_{k+1}(qz/v_\iota)_k}{(zv_\iota;q)_{k+1}(qz/u_\iota)_k}
  \bigg(\frac{abcd}{q^{N+1}}\bigg)^k
  \nnm\\\nnm&&\xqdn\:+\sum_{k=0}^{\infty}(1-z^{-2})(1-q^{2k+1}z^{-2})
 \frac{(q/az,q/bz,q/cz,q/dz;q)_k}{(a/z,b/z,c/z,d/z;q)_{k+1}}\\
 &&\xqdn\:\times\prod_{\iota=1}^n\frac{(u_\iota/z;q)_{k+1}(q/zv_\iota)_k}{(v_\iota/z;q)_{k+1}(q/zu_\iota)_k}
  \bigg(\frac{abcd}{q^{N+1}}\bigg)^k. \label{left-member}
 \emn
Therefore $f(z)$ is regular within $0<|z|<\infty$ and can be
expanded into a laurent series at $z=0$. Letting $z=e^{i\theta}$ in
\eqref{left-member} with $\theta\in \mathbb{R}$ and then integrating
$f(z)$ over $-\pi\leq\theta\leq\pi$, it is not difficult to see that
 \bnm
 \int_{-\pi}^{\pi}f(e^{i\theta})d\theta=2\pi[z^0]f(z)=4\pi\sum_{k=0}^{\infty}\bigg(\frac{abcd}{q^{N+1}}\bigg)^k
 =\frac{4\pi}{1-abcd/q^{N+1}}.
 \enm
 Instead, we derive from the expression on the right hand side of \eqref{integral-before} that
 \bnm
 \int_{-\pi}^{\pi}f(e^{i\theta})d\theta&&\xqdn=2\int_{0}^{\pi}\frac{h(\cos2\theta;1)}{h(\cos\theta;a,b,c,d)}
\prod_{\iota=1}^n\frac{h(\cos\theta;u_\iota)}{h(\cos\theta;v_\iota)}\,d\theta
 \\&&\xqdn\times\:
\frac{(q,ab,ac,ad,bc,bd,cd;q)_{\infty}}{(abcd/q;q)_{\infty}}
\prod_{\iota=1}^n\ffnk{ccccc}{q}{dv_\iota,qv_\iota/d}{du_\iota,qu_\iota/d}_{\infty}
 \\&&\xqdn\times\:
 \sum_{\tilde{m}}\ffnk{ccccc}{q}{v_n/u_n}{q}_{m_n}
  \ffnk{ccccc}{q}{q/ad,q/bd,q/cd}{q/du_n,qv_n/d,q^2/abcd}_{M_n}q^{M_n}
 \\&&\xqdn\times\:
 \prod_{s=1}^{n-1}\ffnk{ccccc}{q}{v_s/u_s}{q}_{m_s}
  \ffnk{ccccc}{q}{qu_{s+1}/d,q/dv_{s+1}}{q/du_s,qv_s/d}_{M_s}
  \bigg(\frac{v_{s+1}}{u_{s+1}}\bigg)^{M_s}.
 \enm
Equating the last two equations, we attain Theorem \ref{thm-d} to
finish the proof.
\end{proof}

When $n=0$, Theorem \ref{thm-e} reduces to \eqref{askey-wilson}
exactly.

Taking $n=1$ in Theorem \ref{thm-e}, we achieve the interesting
result under the replacement $v_1\to u, N_1\to n$.

\begin{corl}  \label{corl-c}
For $|abcd/q^{n+1}|<1$, there holds the
 generalization of the Askey-Wilson integral:
 \bnm
\int_{0}^{\pi}\frac{h(\cos2\theta;1)h(\cos\theta;uq^n)}{h(\cos\theta;a,b,c,d,u)}\,d\theta
&&\xqdn=\:\frac{2\pi}{1-abcd/q^{n+1}}\frac{(abcd/q;q)_{\infty}}{(q,ab,ac,ad,bc,bd,cd;q)_{\infty}}
\\&&\xqdn\times\:\frac{1}{(du,qu/d;q)_n}\Bigg/{_4\phi_3}\ffnk{ccccccc}
 {q;q}{q^{-n},q/ad,q/bd,q/cd}{q^{1-n}/du,qu/d,q^2/abcd}.
 \enm
\end{corl}

Letting $d\to q/a,\{N_\iota\}_{\iota=1}^n\to\{m_\iota\}_{\iota=1}^n,
N\to m$ in Theorem \ref{thm-e}, we recover the known result due to
Chu and Ma \citu{chu-b}{Theorem 5}.

\begin{corl}  \label{corl-e}
For $u_\iota/v_\iota=q^{m_\iota}$ and $|bc/q^{m}|<1$ with
$m_\iota\in\mathbb{N}_0$ and $m=\sum_{\iota=1}^nm_\iota$, there
holds the integral formula:
 \bnm
&&\xqdn\int_{0}^{\pi}\frac{h(\cos2\theta;1)}{h(\cos\theta;a,q/a,b,c)}
\prod_{\iota=1}^n\frac{h(\cos\theta;u_\iota)}{h(\cos\theta;v_\iota)}\,d\theta\\
&&\xqdn=\:\frac{2\pi/(1-q^{-m}bc)}{(q,q,ab,ac,qb/a,qc/a;q)_{\infty}}
\prod_{\iota=1}^n\ffnk{ccccc}{q}{au_\iota,qu_\iota/a}{av_\iota,qv_\iota/a}_{\infty}.
 \enm
\end{corl}

\textbf{Acknowledgments}

The work is supported by the Natural Science Foundations of China
(Nos. 11201241 and 11201291).
%%%%%%%%%%%%%%%%%%%%%%%%%%%%%%% %%%%%%%%%%%%%%%%%%%%%%%%%%%%%%%%%%%%
%%%%%%%%%%%%%%%%%%%%%%%%%%%%%%%%%%%%%%%%%%%%%%%%%%%%%%%%%%%%%%%%%%%
%%%%%%%%%%%%%%%%%%%%%%%%%%%%%%%%%%%%%%%%%%%%%%%%%%%%%%%%%%%%%%%%%%%

%%%%%%%%%%%%%%%%%%%%%%%%%%%%%%%%%%%%%%%%%%%%%%%%%%%%%%%%%%%%%%%%%%%
%%%%%%%%%%%%%%%%%%%%%%%%%%%%%%%%%%%%%%%%%%%%%%%%%%%%%%%%%%%%%%%%%%%
%%%%%%%%%%%%%%%%%%%%%%%%%%%%%%%%%%%%%%%%%%%%%%%%%%%%%%%%%%%%%%%%%%%

\end{document}